\documentclass[oneside,english,11pt]{amsart}
\usepackage{lmodern}
\usepackage[T1]{fontenc}       			 % Text encoding, output
\usepackage[latin9]{inputenc}  			 % Input accented characters
\usepackage[english]{babel}			 % u.a. correct hyphenation -, --
\usepackage[top=3cm, bottom=3cm, left=3.5cm, right=3.5cm, heightrounded, marginparwidth=2.5cm, marginparsep=6mm]{geometry}
\usepackage[foot]{amsaddr}
\usepackage{setspace} 
\usepackage{verbatim}
\usepackage{mathrsfs}	
\usepackage{amstext}
\usepackage{amsthm}
\usepackage{amssymb}
\usepackage{amsopn}
\usepackage{bbm}
\usepackage{stix}
\usepackage{enumitem}				% Enumeration
\usepackage[all]{xy}    				% Commutative diagrams
\usepackage{mathdots}   				% Variously oriented dots
\usepackage{aliascnt}   				% Single enumeration of all theorem environments
\usepackage{tikz}  \usetikzlibrary{matrix}	% Vector images
\usepackage{hyperref}				% References with links
\hypersetup{
    colorlinks,
    linkcolor={red!50!black},
    citecolor={green!50!black},
    urlcolor={blue!80!black},
    linktocpage
}

\numberwithin{equation}{section}
\numberwithin{figure}{section}

\theoremstyle{theorem}
  \newtheorem*{thm*}{Theorem}
  \newtheorem*{cor*}{Corollary}
  \newtheorem*{lem*}{Lemma}
  \newtheorem*{ass*}{Standing assumption}
  \newtheorem*{fact*}{Fact}
  \newtheorem*{conj*}{Conjecture}
  \newtheorem*{ques*}{Question}
  
  \newtheorem{thm}{Theorem}[section]

  \newaliascnt{lem}{thm}
  \newtheorem{lem}[lem]{Lemma}
  \aliascntresetthe{lem}

  \newaliascnt{cor}{thm}
  \newtheorem{cor}[cor]{Corollary}
  \aliascntresetthe{cor}

  \newaliascnt{prop}{thm}  
  \newtheorem{prop}[prop]{Proposition}
  \aliascntresetthe{prop}

  \newaliascnt{exmpl}{thm}
  
  \aliascntresetthe{exmpl}

  \newtheorem{thmx}{Theorem}

  \newaliascnt{corx}{thmx}
  
  \aliascntresetthe{corx}

\theoremstyle{remark}
  \newtheorem*{rem*}{Remark}

  \newaliascnt{rem}{thm}
  \newtheorem{rem}[rem]{Remark}
  \aliascntresetthe{rem}

 \theoremstyle{definition}
    \newtheorem*{defn*}{Definition}
    
   \newaliascnt{defn}{thm}
   
   \aliascntresetthe{defn}

\newcommand{\bbC}{\mathbb{C}}

\newcommand{\bbR}{\mathbb{R}}

\newcommand{\bbZ}{\mathbb{Z}}       

\newcommand{\frg}{\mathfrak{g}}
\newcommand{\frh}{\mathfrak{h}}
\newcommand{\frk}{\mathfrak{k}}
\newcommand{\frl}{\mathfrak{l}}
\newcommand{\frmm}{\mathfrak{m}}
\newcommand{\frp}{\mathfrak{p}}

\newcommand{\frz}{\mathfrak{z}}

\DeclareMathOperator{\id}{id}          
\DeclareMathOperator{\HH}{H}

\newcommand{\Hc}{{\rm H}_{\rm c}}

\newcommand{\bcdot}{{\scriptstyle \bullet}}

\DeclareMathOperator{\SL}{SL}         % special linear group
\DeclareMathOperator{\PSL}{PSL}       % projective special linear group
\DeclareMathOperator{\SU}{SU}             % special unitary group
\newcommand{\fsl}{\mathfrak{sl}}		% Lie algebra of SL
\newcommand{\fsu}{\mathfrak{su}}		% Lie algebra of SU
\newcommand{\fso}{\mathfrak{so}}		% Lie algebra of SO

\newcommand{\qand}{\quad \mathrm{and} \quad}

\title{On the Third-Degree Continuous Cohomology of Simple Lie Groups}
\author{Carlos De la Cruz Mengual}
\address{Department of Mathematics, ETH Zurich. R\"amistrasse 101, 8092 Zurich, Switzerland.}
\email{carlos.delacruz@math.ethz.ch}

\begin{document}

\maketitle

\begin{abstract}
We show that the class of connected, simple Lie groups that have non-vanishing third-degree continuous cohomology with trivial $\bbR$-coefficients consists precisely of all simple complex Lie groups and of $\widetilde{\SL_2(\bbR)}$. 
\end{abstract}

%%%%%%%%%%%%%%%%%%%%%%%%%%%%%%%%%%%%%%%%%%%%%%%%%%%%%%%%%%%%%%
\section{Introduction} \label{sec:introd} 
Continuous cohomology $\Hc^\bcdot$ is an invariant of topological groups, defined in an analogous fashion to the ordinary group cohomology, but with the additional assumption that cochains---with values in a topological group-module as coefficient space---are continuous. A basic reference in the subject is the book by Borel--Wallach \cite{BW}, while a concise survey by Stasheff \cite{Stas} provides an overview of the theory, its relations to other cohomology theories for topological groups, and some interpretations of algebraic nature for low-degree cohomology classes. \vspace{4pt}

In the case of connected, (semi)simple Lie groups and trivial real coefficients, continuous cohomology is also known to successfully detect geometric information. For example, the situation for degree two is well understood: Let $G$ be a connected, non-compact, simple Lie group with finite center. Then $\Hc^2(G;\bbR) \neq 0$ if and only $G$ is of Hermitian type, i.e. if its associated symmetric space of non-compact type admits a $G$-invariant complex structure. If that is the case, $\Hc^2(G;\bbR)$ is one-dimensional, and explicit continuous 2-cocycles were produced by Guichardet--Wigner in \cite{GW} as an obstruction to extending to $G$ a homomorphism $K \to S^1$, where $K < G$ is a maximal compact subgroup. \vspace{4pt}

The goal of this note is to clarify a similar geometric interpretation of the third-degree continuous cohomology of simple Lie groups. Recall that a complex structure on a Lie algebra $\frg$ is a linear map $J\in {\rm End}(\frg)$ that satisfies the identity $J^2 = -\id$ and that commutes with the adjoint representation of $\frg$.

\begin{thmx} \label{thm:simplefc}
Let $G$ be a connected, simple Lie group with finite center, and let $\frg$ be its Lie algebra. Then the following are equivalent:
\begin{enumerate}[label=\emph{(\arabic*)}]
	\item $\Hc^3(G;\bbR)\neq 0$. \vspace{1pt}
	\item $\dim \Hc^3(G;\bbR) = 1.$
        \item $\frg$ admits a complex structure.
        \item $G$ admits the structure of a complex Lie group.
\end{enumerate}
\end{thmx}

%Hence, a group $G$ as in the statement of \autoref{thm:simplefc} has non-vanishing $\Hc^3$ if and only if it is a finite cover of a member of one of the classical families $A_n=\PSL_{n+1}(\bbC)$, $B_n=\SO_{2n+1}(\bbC)$, $C_n={\rm P}\!\Sp_{2n}(\bbC)$ and $D_n={\rm P}\!\SO_{2n}(\bbC)$ or of one of the exceptional complex Lie groups $E_6$, $E_7$, $E_8$, $F_4$ or $G_2$. 
\noindent Removing the hypothesis of finite center results in the addition of only one Lie group to this collection. 

\begin{thmx} \label{thm:simpleic}
For a connected, simple Lie group $G$ of infinite center, $\Hc^3(G;\bbR) \neq 0$ if and only if $G$ is isomorphic to $\widetilde{\SL(2,\bbR)}$, the universal cover of $\SL(2,\bbR)$. In that case, $\dim \Hc^3(G;\bbR) = 1$. 
\end{thmx}

Before proceeding to the proofs of the theorems, we comment on the question of explicit 3-cocycles in the setting of \autoref{thm:simplefc}. Thus, fix $G$ as in \autoref{thm:simplefc}, and assume that the equivalent conditions hold. Let $J$ be a complex structure on the Lie algebra $\frg$ of $G$, and regard it as a complex Lie algebra. Let 
\begin{itemize}
 \item $\frk \subset \frg$ be a compact real form of $\frg$, 
 \item $B_\frg$ be the Killing form of $\frg$, and 
 \item $K$ the connected subgroup of $G$ with Lie algebra $\frk$.
\end{itemize}
Then $\frg = \frk \oplus J\!\frk$ is a Cartan decomposition, and $\frk$ is simple; see \autoref{thm:cartandecJ} below for a reference. The subgroup $K$ is maximal compact in $G$, and $G/K$ is a symmetric space of non-compact type with ${\rm T}_K(G/K) \cong \frg/\frk \cong J\!\frk$. Let $(\Lambda^\bcdot(\frg/\frk)^\ast)^\frk$ denote the complex of $\frk$-invariant, alternating, multilinear forms on $\frg/\frk$, and $\Omega^\bcdot(G/K)^G$ be the complex of $G$-invariant differential forms on $G/K$. Left-translation of an element of the former gives rise of an element of the latter, and this assignment is an isomorphism. \vspace{4pt}

It is a consequence of van Est's theorem that there is an isomorphism $\Hc^3(G;\bbR) \cong (\Lambda^\bcdot(J\!\frk)^\ast)^\frk$; use \autoref{thm:fc} with $\frp = J\!\frk$. The formula 
\begin{equation} \label{eq:cocycle}
	\omega(X,Y,Z) := B_\frg(X,J[Y,Z]), \qquad X,Y,Z \in J\!\frk,
\end{equation}
defines a non-zero element $\omega \in (\Lambda^3(J\!\frk)^\ast)^\frk$. Let $\tilde{\omega} \in \Omega^3(G/K)^G$ be corresponding 3-form. This one integrated over ``3-simplices'' produces a non-trivial, $G$-invariant continuous 3-cocycle $I_{\omega}: G^4 \to \bbR$. \vspace{4pt}

More precisely: Fix a base point $o \in G/K$. For any $k$-tuple $(g_0,\ldots,g_k) \in G^k$, consider the \emph{geodesic $k$-simplex} $\Delta(g_0,\ldots,g_k) \subset G/K$, defined inductively as follows: let $\Delta(g_0) := \{g_0 \cdot o\}$, and for $k > 0$, set $\Delta(g_0,\ldots,g_k)$ to be the union of the geodesics connecting $g_k \cdot o$ to each point in $\Delta(g_0,\ldots,g_{k-1})$.  It is not hard to verify that the expression 
\begin{equation} \label{eq:integral}
	I_{\omega}(g_0,\ldots,g_4) := \int_{\Delta(g_0,\ldots,g_4)} \tilde{\omega}
\end{equation}
is a well-defined $G$-invariant continuous 3-cocycle. Finally, its non-triviality follows from the fact, by Dupont \cite{Dup}, that integration over simplices realizes van Est's isomorphism at the level of cochains. \vspace{4pt} 

\begin{rem}
Concerning \autoref{thm:simpleic}, the group $G := \widetilde{\SL(2,\bbR)}$ has trivial maximal compact subgroup $M$. Thus, van Est's theorem (\autoref{thm:vanest} below) yields an isomorphism $\Hc^3(G;\bbR) \cong \HH^3(\Omega^\bcdot(G)^G)$. An obvious $G$-invariant 3-form of $G$ is its volume form. Arguing as with \eqref{eq:integral}, we conclude that the volume of 3-simplices is a non-trivial $G$-invariant continuous 3-cocycle of $G$. 
\end{rem}

We point out that the expression obtained by removing the $J$ in \eqref{eq:cocycle} is known to define a generating class of $\HH^3(\frk;\bbR)$ and of the de Rham cohomology group $\HH^3(K;\bbR)$; see the reference \cite{GHV}. However, we have found no account in the literature of the formula \eqref{eq:cocycle} nor of the statements of our two theorems. This was in fact the main motivation for writing this note. \vspace{4pt}

The rest of this paper contains proofs of Theorems \ref{thm:simplefc} and \ref{thm:simpleic}. They do not rely on the classification of simple Lie groups. \vspace{7pt}

\noindent {\bf Acknowledgments:} The author expresses his gratitude to his PhD advisor Marc Burger for suggesting the question of the characterization that led to \autoref{thm:simplefc}, and for many helpful discussions. He further thanks Tobias Hartnick and Maria Beatrice Pozzetti for their valuable comments and suggestions.
%\vspace{4pt}

%%%%%%%%%%%%%%%%%%%%%%%%%%%%%%%%%%%%%%%%%%%%%%%%%%%%%%%%%%%%%%%
\section{Notation and background} 

\subsection{Notation} Whenever there is no explicit mention of coefficients when using any notion of cohomology, it should be understood for the rest of this paper that they are trivial $\bbR$-coefficients. The functor $\Hc^\bcdot$ refers to the continuous cohomology of topological groups. On the other hand, when applied to manifolds (including Lie groups), $\HH^\bcdot$ denotes their cohomology as spaces. \vspace{4pt}

If $\frg$ is a Lie algebra and $\frmm \subset \frg$ is any subalgebra, then $(\Lambda^\bcdot(\frg/\frmm)^\ast)^\frmm$ will denote the complex of $\frmm$-invariant, alternating, multilinear forms on $\frg/\frmm$; for the definition of the $\frmm$-action on $\frg/\frmm$ and of the coboundary operator, we refer the reader to Chapter 1 of \cite{BW}. The cohomology of this complex, denoted by $\HH^\bcdot(\frg,\frmm)$ is known as the \emph{Lie algebra cohomology of $\frg$ relative to $\frmm$}. If $\frmm$ is trivial, we will write $\HH^\bcdot(\frg)$ instead of $\HH^\bcdot(\frg,\frmm)$. \vspace{4pt}

If $G$ is a connected Lie group and $M < G$ is a closed subgroup, then $\Omega^\bcdot(G/M)^G$ denotes the complex of $G$-invariant differential forms on $G/M$, where the coboundary operator is the usual differential for forms, and the $G$-action on forms on $G/M$ is by left-translation. We denote the cohomology of this complex by $\HH^\bcdot_G(G/M)$; it is known as the \emph{$G$-invariant de Rham cohomology of $G/M$}. \vspace{4pt}

\subsection{Background on continuous cohomology of Lie groups} A powerful tool for computing the continuous cohomology of a connected Lie group $G$ is \emph{van Est's theorem}; see the original references by van Est \cite{vEst} and Hochschild--Mostow \cite{HM}. We quote it from \cite[Corollary IX.5.6]{BW}.

\begin{thm}[van Est, Hochschild--Mostow] \label{thm:vanest}
Let $G$ be a connected Lie group with Lie algebra $\frg$, and $M$ be a maximal compact (connected) subgroup of $G$ with Lie algebra $\frmm \subset \frg$. Then
\begin{equation} \label{eq:vanest}
	\Hc^\bcdot(G) \cong \HH^\bcdot_G(G/M) \cong \HH^\bcdot(\frg,\frmm). \vspace{10pt}
\end{equation}
\end{thm}

For the rest of this section, let $G$ be a non-compact, connected, semisimple\footnote{For background in semisimple Lie groups and Lie algebras, we refer the reader to \cite{Hel}.} Lie group with Lie algebra $\frg$, and let $M < G$ be a maximal compact subgroup with Lie algebra $\frmm$. Let us fix:
\begin{itemize}
	\item a maximal compactly embedded subalgebra $\frk \subset \frg$ such that $\frmm \subset \frk$,
	\item an associated Cartan decomposition $\frg = \frk \oplus \frp$ of $\frg$,
	\item the connected subgroup $K$ of $G$ with Lie algebra $\frk$,
	\item the compact dual $\frg_u := \frk \oplus i\frp$ of $\frg$,
	\item the \emph{compact dual} $G_u$ of $G$, i.e. the 1-connected, compact, semisimple Lie group with Lie algebra $\frg_u$, and
	\item the connected subgroup $M_u$ of $G_u$ with Lie algebra $\frmm$.
	\item the connected subgroup $K_u$ of $G_u$ with Lie algebra $\frk$.
\end{itemize}
\begin{rem} \label{rem:Kclosed} 
Any subgroup $K' < G$ with Lie algebra $\frk$ is automatically connected and closed in $G$, and contains the center of $G$. Furthermore, $K'$ is compact if and only if the center of $G$ is finite. If that is the case, $K'$ is a maximal compact subgroup of $G$, and $G/K'$ is a symmetric space of non-compact type. This is the content of Theorem VI.1.1 in \cite{Hel}.
\end{rem}
\begin{rem} \label{rem:Kuclosed} 
The connected subgroup $K_u < G_u$ above is necessarily closed in $G_u$. This and the 1-connectedness of $G_u$ imply that $G_u/K_u$ is a symmetric space of compact type. This follows from Proposition IV.3.6 in \cite{Hel}. \vspace{4pt}
\end{rem}

A second computational tool is the next theorem of Chevalley--Eilenberg \cite{ChEil}, whose main ideas they attribute to Cartan. \vspace{-1pt}
\begin{thm}[Cartan, Chevalley--Eilenberg] \label{thm:cartan}
If $M_u < G_u$ is closed, then $\HH^\bcdot(\frg_u,\frmm) \cong \HH^\bcdot(G_u/M_u)$. 
\end{thm}
\noindent It is not hard to observe that there is an isomorphism $\HH^\bcdot(\frg,\frmm) \cong \HH^\bcdot(\frg_u,\frmm)$. Combining it with \eqref{eq:vanest}, we obtain:
\begin{cor} \label{thm:corollary_cartan}
If $M_u < G_u$ is closed, then $\Hc^\bcdot(G) \cong \HH^\bcdot(G_u/M_u)$. \vspace{5pt}
\end{cor}

\noindent \textbf{The case of finite center.} We impose now the additional assumption that $G$ has a finite center. Then, by \autoref{rem:Kclosed}, the subgroup $K<G$ fixed above is maximal compact. In particular, 
\begin{equation} \label{eq:fc}
	M=K, \quad \frmm=\frk \qand M_u = K_u.
\end{equation}
Moreover, by the same remark, $G/K$ is a symmetric space of non-compact type. It is a theorem by Cartan that invariant differential forms on a symmetric space are automatically closed. Hence:

\begin{cor} \label{thm:fc}
Under the assumption of finite center of $G$, there exist isomorphisms
\[
	\Hc^\bcdot(G) \cong \Omega^\bcdot(G/K)^G \cong (\Lambda^\bcdot \frp^\ast)^\frk \cong (\Lambda^\bcdot (i\frp)^\ast)^\frk \cong \HH^\bcdot(G_u/K_u). \vspace{8pt}
\]
\end{cor}

%%%%%%%%%%%%%%%%%%%%%%%%%%%%%%%%%%%%%%%%%%%%%%%%%%%%%%%%%%%%%%%
\section{Proof of \autoref{thm:simplefc}} 
The equivalence (3) $\Leftrightarrow$ (4) is a classical fact that holds for any connected Lie group, and goes back to Newlander--Nirenberg. Thus, we will only prove the equivalence (1) $\Leftrightarrow$ (2) $\Leftrightarrow$ (3). A first reduction is provided by the following lemma.

\begin{lem}
If $G$ is a compact, connected, simple Lie group (hence with finite center), then none of the statements \emph{(1)-(4)} in \autoref{thm:simplefc} hold. 
\end{lem}
\begin{proof}
If $G$ satisfied (4), then by connectedness and compactness, $G$ would be isomorphic to a quotient $\bbC^n/\Lambda$, where $\Lambda < \bbC^n$ is a lattice. In particular, $G$ would be Abelian, contradicting the assumption of simplicity. On the other hand, properties (1) and (2) fail because by \autoref{thm:vanest}, the continuous cohomology of a compact Lie group vanishes in any positive degree.
\end{proof}

Hence, from now on in this section, let $G$ be a \emph{non-compact}, connected, simple Lie group with finite center, and let $\frg$ be its Lie algebra. We argue now in the order (3) $\Rightarrow$ (2) $\Rightarrow$ (1) $\Rightarrow$ (3), where (2) $\Rightarrow$ (1) is evident. 

\subsection{Proof of (3) $\Rightarrow$ (2)} 
Assume that $\frg$ admits a complex structure $J$. Then:
\begin{thm} \label{thm:cartandecJ}
Any compact real form $\frk$ of $\frg$ is simple, maximal compactly embedded in $\frg$, and
\begin{equation} \label{eq:cartandec}
	\frg = \frk \oplus J\!\frk
\end{equation}
is a Cartan decomposition of $\frg$.
\end{thm}
\begin{proof}[About the proof]
The fact that $\frk$ is maximal compactly embedded and that the decomposition above is a Cartan decomposition is Corollary III.7.5 of \cite{Hel}. The simplicity of $\frk$ is part of the proof of Theorem VIII.5.4 in \cite{Hel}.
\end{proof}

Fix a compact real form $\frk$ of $\frg$. By \autoref{thm:fc} (with $\frp = J\!\frk$), there exists isomorphisms $\Hc^3(G) \cong (\Lambda^3 (J\!\frk)^\ast)^\frk \cong (\Lambda^3 \frk^\ast)^\frk$. Let $({\rm V}^2 \frk^\ast)^\frk$ denote the space of $\frk$-invariant, symmetric bilinear forms on $\frk$. We make use of the following fact:
\begin{prop}
The assignment $\Phi: ({\rm V}^2 \frk^\ast)^\frk \to (\Lambda^3 \frk^\ast)^\frk$, defined by
\[
	\Phi_B(X,Y,Z) = B(X,[Y,Z]), \quad \mbox{for } B \in ({\rm V}^2 \frk^\ast)^\frk \mbox{ and } X,Y,Z \in \frk,
\]
is a linear isomorphism.
\end{prop}
\begin{proof}[About the proof]
This is Proposition I in Section 5.7 of \cite{GHV}. The statement holds if $\frk$ is any compact, simple Lie algebra. 
\end{proof}
\vspace{0.2cm}
We conclude the proof of the implication by pointing out that the space $({\rm V}^2 \frk^\ast)^\frk$ has dimension one because of the simplicity of $\frk$; a generator of it is the Killing form of $\frk$. A proof of this fact can be found, for example, in the discussion  at the beginning of VIII.\S 5 of \cite{Hel}. 
\vspace{0.6cm}

\subsection{Proof of (1) $\Rightarrow$ (3)}
Assume now that $\frg$ does not admit a complex structure. We adopt the same notation of \eqref{eq:fc} and \autoref{thm:fc}, and fix:
\begin{itemize}
        \item a maximal compact subgroup $K < G$,
        \item its Lie algebra $\frk \subset \frg$,
	\item an associated Cartan decomposition $\frg = \frk \oplus \frp$ of $\frg$,
	\item the compact dual $\frg_u := \frk \oplus i\frp$ of $\frg$,
	\item the compact dual $G_u$ of $G$, and
	\item the connected subgroup $K_u$ of $G_u$ with Lie algebra $\frk$. \vspace{10pt}
\end{itemize}
The starting point for the proof of this implication is the following theorem.  
\begin{thm} \label{thm:gu_simple}
\hspace{-0.7cm} \begin{minipage}[t]{12.9cm}
\vspace{-8pt} \begin{enumerate}[label=\emph{(\roman*)}]
\item The compact dual $\frg_u$ is simple.
\item The symmetric space of compact type $G_u/K_u$ is irreducible. In particular, the action of  ${\rm Ad}(K_u)$ on the vector space $i\frp$ is irreducible.
\end{enumerate} 
\end{minipage}
\end{thm}
\begin{proof}[About the proof]
Part (i) is a combination of Theorem VIII.5.3 and Theorem V.2.4 of \cite{Hel}. Part (ii) is a consequence of Theorem VIII.5.3 of \cite{Hel}; see also the Definition at the beginning of VIII.\S 5 of \cite{Hel}.
\end{proof}

\begin{rem}
Note that (i) does not hold for Lie algebras that do admit a complex structure. For example, the compact dual of $\fsl(2,\bbC)$ is the Lie algebra $\fsu(2) \oplus \fsu(2) \cong \fso(4,\bbR)$. \vspace{4pt}
\end{rem}

By \autoref{thm:fc}, we have $\Hc^3(G) \cong \HH^3(G_u/K_u)$. Thus it suffices to show that $\HH^3(G_u/K_u)$ vanishes. We distinguish two cases: \vspace{6pt}

\noindent \textbf{Case 1: $\frk$ is Abelian.} The following proposition establishes the claim.

\begin{prop} \label{thm:abelian}
The symmetric space $G_u/K_u$ is diffeomorphic to the 2-dimensional sphere $S^2$.  
\end{prop}

\begin{proof}
The group $K_u$ is Abelian; because it is a non-trivial torus, it contains an element $j$ of order four. The image of $j$ under the adjoint representation ${\rm Ad}_{G_u}(j)|_{i\frp}$ is a complex structure on the real vector space $i\frp$. Thus, we regard $i\frp$ now as a $\bbC$-vector space. By the commutativity of $K_u$, the ${\rm Ad}(K_u)$-action on $i\frp$ is $\bbC$-linear. By (ii) of \autoref{thm:gu_simple} and Schur's lemma, we obtain that $\dim_\bbC i\frp = 1$.  In consequence,
\[
	\dim(G_u/K_u) = \dim_\bbR i\mathfrak{p}= 2 \dim_\bbC i\frp = 2.
\]
Since $G_u$ is 1-connected and $K_u$ is connected, the space $G_u/K_u$ is 1-connected, and we conclude by the classification of surfaces.
\end{proof}
\vspace{0.2cm}
\noindent \textbf{Case 2: $\frk$ is non-Abelian.} The claim follows from the next proposition with $U = G_u$ and $L = K_u$.

\begin{prop} \label{thm:nonabelian}
Let $U$ be a 1-connected, compact, simple Lie group, and $L < U$ be a connected, closed, non-Abelian subgroup. Then $\HH^3(U/L) = 0$.
\end{prop}

\noindent We point out that $U/L$ needs not be a symmetric space. We give a proof of this proposition in the next subsection. \vspace{0.4cm}

\subsection{Proof of \autoref{thm:nonabelian}} We will need three facts about the topology of Lie groups.

\begin{thm}[Weyl] \label{thm:weyl}
The universal covering group of a compact semisimple Lie group is compact. In particular, any Lie group with a compact, semisimple Lie algebra is compact. 
\end{thm} 
\begin{proof}[About the proof]
A proof of this theorem is found in \cite{Hel} as Theorem II.6.9.
\end{proof}
\vspace{0.1cm}
\begin{thm}[Bott] \label{thm:pi23} 
For any connected Lie group $H$, the second homotopy group $\pi_2(H)$ vanishes. If, moreover, $H$ is simple, then $\pi_3(H) \cong \bbZ$. 
\end{thm}

\begin{proof}[About the proof]
It is a consequence of the fact, due to Bott, that the loop space of a compact, 1-connected Lie group has the homotopy type of a CW-complex with no odd-dimensional cells and finitely many cells of each even dimension. A proof of this fact is completely Morse-theoretical; see Theorem 21.6 in \cite{Mil}. Based on it, one can conclude as indicated in \cite{MO}.
\end{proof}
\vspace{0.2cm}
Before quoting the third fact, we give the definition of the \emph{Dynkin index} of a homomorphism between two compact, connected, simple Lie groups, as found in \cite{Oni}: 

\begin{defn*}
Let $G_1$ and $G_2$ be compact, connected, simple Lie groups with Lie algebras $\frg_1$ and $\frg_2$, respectively, and let $\alpha_i \in\frh_i^\ast$ be a root of maximal length of $\frg_i$ with respect to a Cartan subalgebra $\frh_i$ ($i=1,2$). Moreover, let $B_i \in ({\rm V}^2 \frg_i^\ast)^{\frg_i}$ be a negative-definite, $\frg_i$-invariant bilinear form\footnote{As mentioned in the proof of the implication (3) $\Rightarrow$ (2) of \autoref{thm:simplefc}, $\dim ({\rm V}^2 \frg_i^\ast)^{\frg_i} = 1$, so the forms $B_i$ are uniquely determined up to a positive constant.}  on $\frg_i$ normalized in such a way that the square of the length of the root $\alpha_i$ with respect to the associated inner product on $\frh_i^\ast$ equals two. If $\varphi: G_1 \to G_2$ is a homomorphism, then there exists a non-negative real number $j_\varphi$ such that $B_2 = j_\varphi \cdot B_1$, called the \emph{Dynkin index} of $\varphi$. 
\end{defn*}

\begin{thm} \label{thm:dynkinindex}
Let $\varphi: G_1 \to G_2$ be a homomorphism between two compact, connected, simple Lie groups $G_1$ and $G_2$. Then:
\begin{enumerate}[label=(\roman*)]
	\item The Dynkin index $j_\varphi$ is a non-negative integer. It is equal to zero if and only if $\varphi$ is the homomorphism that maps every element to the identity.
	\item If $\pi_3(G_i) = \langle \epsilon_i \rangle$ for $i=1,2$ (see \autoref{thm:pi23}), then $\varphi_\# \, \epsilon_1 = \pm j_\varphi \, \epsilon_2$. 
\end{enumerate}
\end{thm}
\begin{proof}[About the proof]
Part (i) is Proposition 11, \S 3 of \cite{Oni}, and (ii) is Theorem 2, \S 17 of the same reference.
\end{proof}

\begin{proof}[Proof of \autoref{thm:nonabelian}]
We will prove that the integral homology group $\HH_3(U/L;\bbZ)$ is finite and therefore by the universal coefficient theorem $\HH^3(U/L)=0$. Because $U$ is 1-connected and $L$ is connected, we have that $U/L$ is 1-connected. Therefore, the Hurewicz homomorphism $h:\pi_3(U/L) \rightarrow H_3(U/L;\bbZ)$ in degree three is surjective. The claim follows after proving finiteness of $\pi_3(U/L)$. \vspace{4pt}

Consider now the long exact sequence in homotopy
\[
\cdots \: \rightarrow \pi_3(L) \xrightarrow{\iota_{\#}} \pi_3(U) \xrightarrow{\pi_{\#}} \pi_3(U/L) \rightarrow \pi_2 (L) \rightarrow \cdots
\]
of the fibration $L \overset{\iota}{\hookrightarrow} U \overset{\pi}{\twoheadrightarrow} U/L$. By \autoref{thm:pi23} and exactness, the homomorphism $\pi_{\#}$ is surjective and $\pi_3(U/L) \cong \pi_3(U)/{\rm im} \, \iota_{\#} \cong \bbZ/{\rm im} \, \iota_\#$. Therefore, $\pi_3(U/L)$ is finite if and only if the image of $\iota_{\#}$ is non-trivial. \vspace{4pt}

Let $\frl$ be the Lie algebra of $L$, hence compact and non-Abelian. In particular, $\frl$ splits as the direct sum of its center and a non-trivial semisimple ideal. Thus, let $\frl_1$ be a simple ideal of $\frl$, and let $L_1$ be the 1-connected, simple Lie group with Lie algebra $\frl_1$, which is compact by \autoref{thm:weyl}. Now let $\phi: L_1 \rightarrow L$ be the unique Lie group homomorphism whose derivative is the inclusion $\frl_1 \hookrightarrow \frl$. Again by \autoref{thm:pi23}, we know that $\pi_3(K_1) \cong \bbZ$. We obtain the following diagram in homotopy: \vspace{-3pt}
\[
	\xymatrix{\pi_3(L) \ar[r]^{\iota_{\#}} & \pi_3(U) \cong \bbZ \\
	\bbZ \cong \pi_3(L_1) \qquad \ar[u]^{\phi_{\#}} \ar@{-->}[ur]_{\iota_{\#} \circ \phi_{\#} =: \psi_\#} &  &  & \:}
\]
Set $\psi:=\iota \circ \phi$. Obviously, the image of $\psi_{\#}=\iota_{\#} \circ \phi_{\#}$ is contained in the image of $\iota_{\#}$. Hence, it suffices to show that the former one is non-trivial in order to prove that so is the latter. \vspace{4pt}

To conclude, note that $\psi$ is an immersion, since its derivative is injective. In particular, it is not the homomorphism that maps every element of $L_1$ to the identity of $U$. Thus, the Dynkin index $j_\psi$ of $\psi$ is not zero by \autoref{thm:dynkinindex} (i). Furthermore, by \autoref{thm:dynkinindex} (ii), the generator $\epsilon_{L_1}$ gets mapped by $\psi_{\#}$ to $\pm j_\psi \, \epsilon_U$, so ${\rm im} \, \psi_{\#} \cong j_{\psi} \bbZ \neq \{0\}$.
\end{proof}
\vspace{0.1cm}

\section{Proof of \autoref{thm:simpleic}} 
Let $G$ be a connected, simple Lie group with infinite center $Z$ and Lie algebra $\frg$. We set $G_0 := G/Z$, which is a connected, center-free Lie group with Lie algebra $\frg$, and let $p: G \twoheadrightarrow G_0$ be the canonical projection. In addition, let 
\begin{itemize}
	\item $\frk$ be a maximal compactly embedded subalgebra of $\frg$, which splits as the direct sum	$\frk = \frz(\frk) \oplus \frmm$ of its center $\frz(\frk)$ and a semisimple or trivial ideal $\frmm$;
	\item $\frg = \frk \oplus \frp$ be a Cartan decomposition of $\frg$,
	\item $\frg_u = \frk \oplus i\frp$ be the compact dual of $\frg$,
	\item $K_0 < G_0$ be the connected subgroup with Lie algebra $\frk$,
	\item $K := p^{-1}(K_0)$,
	\item $G_{0,u}$ denote the compact dual of $G_0$,
	\item $K_{0,u}$ be the connected subgroup of $G_{0,u}$ with Lie algebra $\frk$, and
	\item $G_u$ be the compact dual of $G$. 
\end{itemize}

Note first that $\frg$ does not admit a complex structure, because complex Lie groups cannot have infinite center. In consequence, by \autoref{thm:gu_simple}, the Lie algebra $\frg_u$ and the compact dual $G_{0,u}$ are simple. Furthermore, being a cover of $K_0$, the Lie group $K$ has also Lie algebra $\frk$. By \autoref{rem:Kclosed}, $K$ is connected, closed and non-compact in $G$. The non-compactness of $K$ implies that $\frz(\frk)$ cannot be trivial: if that were the case, then $\frk = \frmm \neq 0$, and $K$ would be semisimple, contradicting \autoref{thm:weyl}. \vspace{4pt}

We distinguish again two cases: $\frk$ Abelian and $\frk$ non-Abelian. We will show that the former assumption corresponds to the situation in which $G$ is isomorphic to $\widetilde{\SL(2,\bbR)}$, and then show that $\Hc^3(G)$ is one-dimensional in that case. Then, we show that the latter implies vanishing of $\Hc^3(G)$. \vspace{0.4cm}

\noindent \textbf{Case 1: $\frk$ is Abelian.} We are exactly in the situation of \autoref{thm:abelian}. Thus, it follows that the symmetric space of non-compact type $G_{0,u}/K_{0,u}$ is diffeomorphic to $S^2$. Then, via duality, $G_0/K_0$ is diffeomorphic to the hyperbolic plane $H^2$, $G_0 \cong \PSL(2,\bbR)$, and $G \cong \widetilde{\SL(2,\bbR)}$, being the only infinite cover of $G_0$. \vspace{4pt}

Before showing that the dimension of $\Hc^3(G)$ equals one, we prove the following lemma, which will also be useful in the case of $\frk$ non-Abelian. 

\begin{lem} \label{thm:Mmaxcpt} If $\dim \frz(\frk) = 1$, then the connected subgroup $M < G$ with Lie algebra $\frmm$ is maximal compact in $G$. 
\end{lem}
\begin{proof}
By the semisimplicity of $\frmm$ and \autoref{thm:weyl}, the Lie subgroup $M < G$ is compact. Since $K$ is non-compact, the connected Lie subgroup $R < K$ with Lie algebra $\frz(\frk)$ must be non-compact as well. The dimension assumption on $\frz(\frk)$ implies that $R$ is isomorphic to $\bbR$. \vspace{4pt}

Note that the intersection $R \cap M$ is a compact subgroup of $R \cong \bbR$, hence trivial. From the decomposition $\frk = \frz(\frk) \oplus \frmm$ and the previous fact, we obtain an isomorphism $K \cong R \times M$. This implies that $M$ is the unique maximal compact subgroup of $K$. It is in fact a maximal compact subgroup of $G$:  If $L <G$ is a compact subgroup containing $M$, then there exists an element $g \in G$ such that $gLg^{-1} < K$. Consequently, $gMg^{-1} < gLg^{-1} < M$. The equalities hold because $gMg^{-1} < M$ and the two are connected subgroups of $K$ with the same Lie algebra. 
\end{proof}

Note that in our case $\frz(\frk) = \frk \cong \fso(2,\bbR)$, which is one-dimensional, and $\frmm$ is trivial. Hence, the connected subgroup $M$ of $G$ with Lie algebra $\frmm$ is therefore trivial, and by the previous lemma, maximal compact in $G$. Moreover, it is well known that the compact dual of $\frg$ is $\frg_u=\fsu(2)$. Thus, $G$ has as compact dual the Lie group $G_u = \SU(2) \cong S^3$. From \autoref{thm:corollary_cartan} with $M_u$ trivial, we have an isomorphism 
\[
	\Hc^3(G) \cong \HH^3(G_u) = \HH^3(S^3),
\]
and the last one is clearly one-dimensional. \vspace{10pt}

\noindent \textbf{Case 2: $\frk$ is non-Abelian.} This assumption means that both $\frz(\frk)$ and $\frmm$ are non-trivial. By the non-triviality of $\frz(\frk)$ and the simplicity of $G_0$, it follows from Theorems VIII.6.1 and VIII.6.2 of \cite{Hel} that $G_0/K_0$ is an irreducible Hermitian symmetric space of non-compact type, and that the center $Z(K_0)$ of $K_0$ is isomorphic to $S^1$. In particular, $\dim \frz(\frk) = 1$. \vspace{4pt}

The connected subgroup $M < G$ with Lie algebra $\frmm \subset \frg$ is non-trivial and semisimple. By \autoref{thm:Mmaxcpt}, it is also a maximal compact subgroup. Thus, by \autoref{thm:vanest}, $\Hc^3(G) \cong \HH^3(\frg,\frmm)$. On the other hand, by \autoref{thm:weyl}, the connected Lie subgroup $M_u < G_u$ with Lie algebra $\frmm$ is closed. We conclude now by \autoref{thm:corollary_cartan} and \autoref{thm:nonabelian}:
\[
	\Hc^3(G) \cong \HH^3(\frg,\frmm) \cong \HH^3(G_u/M_u) = 0. \vspace{8pt}
\]


\begin{thebibliography}{99}

\bibitem{BW} A. Borel, N. Wallach. \emph{Continuous cohomology, discrete subgroups, and representations of reductive groups}. Second edition. Mathematical Surveys and Monographs, 67. American Mathematical Society, Providence, RI, 2000. xviii+260 pp. ISBN: 0-8218-0851-6 

%\bibitem{BS} R. Bott, H. Samelson. \emph{Applications of the theory of Morse to symmetric spaces}. Amer. J. Math. 80, 964-1029 (1958).

\bibitem{ChEil} C. Chevalley, S. Eilenberg. \emph{Cohomology theory of Lie Groups and Lie algebras}. Trans. Amer. Math. Soc. 63, 85-124 (1948).

\bibitem{MO} J. DeVito (https://mathoverflow.net/users/1708/jason-devito). \emph{Homotopy groups of Lie groups}. URL (version: 2016-07-15): https://mathoverflow.net/q/8961

\bibitem{Dup}  J. L. Dupont.  \emph{Simplicial de Rham cohomology and characteristic classes of flat bundles}. Topology 15 (1976), no. 3, 233-245.

\bibitem{vEst} W. T. van Est. \emph{On the algebraic cohomology concepts in Lie groups}, I, II. Nederl. Akad. Wetensch. Proc. Ser. A. 58 (1955). 225-233, 286-294.

\bibitem{GW} A. Guichardet, D. Wigner. \emph{Sur la cohomologie r\'eelle des groupes de Lie simples r\'eels} (French. English summary.) Ann. Sci. \'Ecole Norm. Sup. (4) 11 (1978), no. 2, 277-292. 

\bibitem{GHV} W. Greub, S. Halperin, R. Vanstone. \emph{Connections, curvature and cohomology, Volume III: Cohomology of principal bundles and homogeneous spaces}. Pure and Applied Mathematics, Vol. 47-III. Academic Press [Harcourt Brace Jovanovich, Publishers], New York--London, 1976. xxi+593 pp. 

\bibitem{Hel} S. Helgason. \emph{Differential Geometry, Lie Groups and Symmetric Spaces}. Corrected reprint of the 1978 original. Graduate Studies in Mathematics, 34. American Mathematical Society, Providence, RI, 2001. xxvi+641 pp. ISBN: 0-8218-2848-7 

\bibitem{HM} G. Hochschild, G. D. Mostow. \emph{Cohomology of Lie groups}. Illinois J. Math. 6 (1962), 367-401.

\bibitem{Mil} J. Milnor. \emph{Morse theory}. Based on lecture notes by M. Spivak and R. Wells. Annals of Mathematics Studies, No. 51, Princeton University Press, Princeton, N.J. 1963 vi+153 pp.

\bibitem{Oni} A. L. Onishchik. \emph{Topology of transitive transformation groups.} Johann Ambrosius Barth Verlag GmbH, Leipzig, 1994. xvi+300 pp. ISBN: 3-335-00355-1 

\bibitem{Stas} J. D. Stasheff. \emph{Continuous cohomology of groups and classifying spaces}. Bull. Amer. Math. Soc. 84 (1978), no. 4, 513-530.

\end{thebibliography}
\end{document}